\documentclass[11pt]{amsart}
\usepackage{geometry}                
\usepackage{graphicx}
\usepackage{amssymb, amsmath, amsfonts, amsthm, textcomp}
\usepackage{ hyperref}

\usepackage{epstopdf}
\usepackage{tikz,tikz-cd}
\newcommand{\Hom}{\operatorname{Hom}}
\newcommand{\Spec}{\operatorname{Spec}}
\newcommand{\Proj}{\operatorname{Proj}}
\newcommand{\enh}{\operatorname{enh}}
\newcommand{\Lie}{\operatorname{Lie}}
\newcommand{\End}{\operatorname{End}}
\newcommand{\id}{\operatorname{id}}

\newtheorem{theorem}{Theorem}[section]
\newtheorem{lemma}[theorem]{Lemma}

\newtheorem{proposition}[theorem]{Proposition}

\title{Group compactifications and moduli spaces}
\author{Johan Martens}
\thanks{The author was partially supported by EPSRC grant EP/N029828/1.}
\address[Johan Martens]{School of Mathematics and Maxwell Institute, The University of Edinburgh, Peter Guthrie Tait Road, Edinburgh EH9 3FD, United Kingdom}
\email{johan.martens@ed.ac.uk}

\date{\today}                                           

\begin{document}
\begin{abstract}
We give a summary of joint work with Michael Thaddeus that realizes toroidal compactifcations of split reductive groups as moduli spaces of framed bundles on chains of rational curves.  We include an extension of this work that covers Artin stacks with good moduli spaces.  We discuss, for complex groups, the symplectic counterpart of these compactifications, and conclude with some open problems about the moduli problem concerned.
\end{abstract}
\maketitle
\section{Introduction}
\subsection{}The basic problem we want to discuss in this note is the following: given a reductive group $G$ (which, as an affine variety, is not complete), how can you think about completions of $G$?  For good measure, we want to restrict ourselves to completions that are well-behaved, which as a minimum will require $G\times G$-equivariance.  For $G=PGL(n)$ this is a very classical question, going back to 19th century enumerative geometry.  We will make no attempt to relate the history of this, but would like to point out the early paper \cite{semple} by Semple, who was likely the first to think of these completions as (in modern language) moduli spaces, which will be central theme in our discussion.  

In general one can think of this question as a special case of spherical embeddings (thinking of $G$ as $(G\times G)/ G$), and use the Luna-Vust theory that classifies these \cite{luna-vust}, which was developed from the early 1980s.  Around the same time, in another seminal work \cite{DCP1,DCP2}, De Concini and Procesi gave a description of embeddings of symmetric varieties, and in particular showed that every \emph{adjoint} group (a reductive group with trivial center, hence necessarily semi-simple) has a canonical smooth completion, which is known as the \emph{wonderful} compactification.  For a general description of equivariant completions of reductive groups, see \cite{timashev}.

\subsection{}An impetus for looking for completions of $G$ as moduli spaces came from the search of a universal compactification of the moduli space of $G$-principal bundles on 
curves, as advocated by Sheshadri.  Indeed, if $X$ is a complete curve with a single nodal singularity whose normalization $\widetilde{X}$ is connected, then any $G$-bundle on $X$ corresponds to a $G$-bundle on $\widetilde{X}$, equipped with an isomorphism between the fibres over the exceptional points of $\tilde{X}$.  This isomorphism can be changed by the action of $G$, which indicates that any suitable compactification of the moduli problem of bundles-on-curves lying over the Deligne-Mumford compactification of the moduli space of curves would give rise to a compactification of the group $G$ itself.
Motivated by this, Kausz developed a particular modular compactification of $GL(n)$ \cite{kausz}.

\subsection{}In \cite{MT1}, Thaddeus and the author introduced a general moduli problem of framed bundle chains.  This is represented by a stack, such that any (torsion-free stacky simplicial) fan with support in the positive Weyl chamber of $G$ determines a stability condition on this stack.  Imposing this condition cuts out a substack, whose coarse moduli space is the toroidal embedding of the group $G$ determined by the fan.  All toroidal embeddings of $G$ can be understood this way.

We will review this work here, indicating how these stacks can be described as global quotients of reductive monoids using work of Vinberg, and show how the simplicial condition can be relaxed, to give Artin moduli stacks with good moduli spaces (in the sense of Alper \cite{alper}).  

Next we show how semi-projective toroidal embeddings of complex groups can be understood in real symplectic geometry as non-abelian cuts of the cotangent bundle of a compact form of the group.  
Finally we conclude with a list of open problems and questions related to the moduli problem we described.  The aim of this note is to provide a readable introduction, for more technical details we refer to \cite{MT1,MT2,MT3}.

\subsection{General setup}
Throughout we will let $G$ be a connected split reductive group over an arbitrary field $k$ (though at some points we will restrict to $k$ algebraically closed, or even complex), with a chosen split maximal torus $T$, center $Z_G$, character lattice $V=\Hom(T,\mathbb{G}_m)$ and co-character lattice $\Lambda=\Hom(\mathbb{G}_m,T)$.  Furthermore we will choose a Borel subgroup $B$, corresponding to positive simple roots $\alpha_i$ and positive Weyl chamber $\Lambda_{\mathbb{Q}}^+\subset \Lambda \otimes_{\mathbb{Z}}\mathbb{Q}$.  

For our purposes a $G$-principal bundle will be a $G$-torsor that is locally trivial in the \'etale topology.  For technical reasons we ask these to be rationally trivial (that is, trivial on the generic point of each component) --- this is automatically satisfied if $k$ is algebraically closed or of characteristic zero.  

A polyhedron will be a subset of a vector space cut out by a finite number of linear inequalities -- so in particular, it need not be bounded.  Finally we shall say that a variety is semi-projective if it is projective over an affine variety.

\section{The moduli problem}
We want to consider the moduli problem of \emph{framed bundle chains}.  
\subsection{}The objects we want to study are $G$-principal bundles over chains of projective lines meeting in nodal singularities. The chains are connected, but their length is a priori arbitrary (though having at least one component) and allowed to vary in families.  We do not want the chain to `flip' in a family though.  To this end, we ask for a consistent choice of \emph{endpoints} $p_+$ and $p_-$ in the chain (which are encoded, for each family, by two sections of the chain, taking values in the extremal components over any closed point).  

\begin{figure}[h]                                                               
\begin{center}                                                                  
\begin{tikzpicture}[scale=1]                                                    
\draw[thick] (1.5, .375) -- (-1.5, -.375);                                
\draw[thick] (.5, .375) -- (3.5, -.375);                                  
\draw[thick] (5.5, .375) -- (2.5, -.375);                                 
\draw[thick] (4.5, .375) -- (7.5, -.375);                                 
\draw[thick] (9.5, .375) -- (6.5, -.375);                                 
\draw[thick] (8.5, .375) -- (11.5, -.375);                                
\filldraw (-1, -.25) circle (1.5pt);                                            
\draw (-1, -.2) node[anchor= south] {$p_+$};                                    
\filldraw (11, -.25) circle (1.5pt);                                            
\draw (11, -.2) node[anchor= south] {$p_-$};                                    
\end{tikzpicture}                                                               
\end{center}                                                                    
\caption{A chain of lines}                                                      
\end{figure}
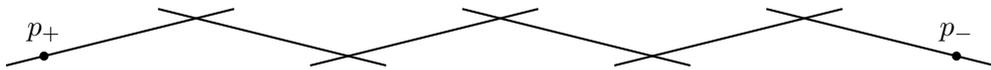

Each such family of chains of projective lines comes with a canonical action of the multiplicative group $\mathbb{G}_m$ \cite[Theorem 1.6]{MT1}.  We ask for all our $G$-bundles to be equivariant with respect to this $\mathbb{G}_m$-action (i.e. to be equipped with the lift of this action to the total space of the bundle, commuting with the $G$-action).  Finally, we want the bundles to be \emph{framed} over $p_+$ and $p_-$, in a $\mathbb{G}_m$-invariant way. That is, the sections $p_+$ and $p_-$ are lifted to the total space of the bundle, and these sections are invariant under the action of $\mathbb{G}_m$.  

\subsection{}Before we can proceed, we need to discuss a description of the isomorphism classes of bundle chains.  Recall that the classical theorem of Birkhoff-Grothendieck-Harder \cite{groth, harder} classifies isomorphism classes of $G$-principal bundles on a single projective line: they are given by elements of $\Lambda/W$.  Naively, but not incorrectly, one can view this theorem as follows: write $\mathbb{P}^1$ by stereographic projection as $\mathbb{A}^1\cup \mathbb{A}^1$ (where the two copies of the affine line are glued together along $\mathbb{A}^1\cap\mathbb{A}^1\cong \mathbb{G}_m)$.    A $G$-principal bundle on $\mathbb{P}^1$ trivialises on each of the $\mathbb{A}^1$s, and is determined by a single clutching function $\mathbb{A}^1\cap \mathbb{A}^1\rightarrow G$. The theorem can be interpreted as saying that this clutching function can always be taken to be a group homomorphism $\rho:\mathbb{G}_m\rightarrow G$, which then can be conjugated to take values in $T\subset G$, and hence gives an element of $\Lambda$.  Finally the action of $W$ gives rise to isomorphic bundles.

We need a generalisation of this in two directions: we want to work with chains of $\mathbb{P}^1$s, not just single copies, and we want to account for the $\mathbb{G}_m$-equivariance.  This generalisation was given in \cite{MT3}.  To summarise it, we can again first look at a single $\mathbb{P}^1$.  By using the $\mathbb{G}_m$-fixed points for stereographic projection, we can again think of $\mathbb{P}^1$ as covered by two copies of $\mathbb{A}^1$, equipped with linear $\mathbb{G}_m$-actions of weights $1$ and $-1$ respectively.  Any $\mathbb{G}_m$-equivariant $G$-principal bundle on $\mathbb{P}^1$ again trivializes on each $\mathbb{A}^1$, and for each the $\mathbb{G}_m$-equivariance can just be characterised by indicating how $\mathbb{G}_m$ acts on the fibre over $p_+$ or $p_-$.  This is given by two elements $\rho_+,\rho_-\in \Lambda$, and it is straightforward to see that the only way to glue two such bundles on $\mathbb{A}^1$ to get a bundle on $\mathbb{P}^1$ is to use $\rho=\rho_+-\rho_-$ as clutching function.  Finally, the diagonal action of $W$ gives again rise to isomorphic bundles, and hence we see that the isomorphism classes of $\mathbb{G}_m$-equivariant $G$-principal bundles are given by $\Lambda^2/W$.

Moreover, for chains this story goes through line by line, and finally we obtain that the isomorphism classes of $\mathbb{G}_m$-equivariant $G$-principal bundles on a chain of $\mathbb{P}^1$s of length $n$ is given by assigning an element of $\Lambda$ to each $\mathbb{G}_m$-fixed point of the chain, and this assignment is again unique up to the diagonal action of $W$.  So we end up with $\Lambda^{n+1}/W$.

\subsection{}Now if we go back to our moduli problem of framed bundle chains, we see that, since our bundles are framed at $p_+$ and $p_-$ (and the framing is preserved by $\mathbb{G}_m$), the co-characters are necessarily zero there, and hence the only (potentially) non-trivially co-characters are assigned to the nodes .  So for a single framed bundle chain over a chain of lenght $n+1$, the isomorphism class of the bundle (forgetting the framing) is given by an element of $\Lambda^n/W$, where the Weyl group $W$ acts diagonally.  We shall refer to this data as the \emph{splitting type} of the bundle.

We shall begin by making the restriction that we only consider bundles where all the elements in $\Lambda^n$ making up its splitting type lie in the same Weyl chamber (this condition is preserved by the action of $W$ of course).  It is shown in \cite[\S 3]{MT1} that this moduli problem determines a smooth Artin stack (non-separated and of infinite type).  Moreover, the group $G\times G$ acts on this stack, by changing the framing over $p_+$ and $p_-$.   Finally, this stack contains $G\times B\mathbb{G}_m$ as the open substack corresponding to those framed bundle chain where all chains have length one. 
To see the last statement, remark that any $\mathbb{G}_m$-equivariant framed bundle on a single projective line is necessarily trivial, and hence the only thing that can vary are the framings.  These can be changed by $G\times G$, but as we only consider these up to isomorphisms, we can quotient out the diagonal $G$.  This gives a family parametrized by $G\cong (G\times G)/G$, but as the $\mathbb{G}_m$-equivariance is built into the moduli problem, we really have the gerbe $G\times B\mathbb{G}_m$.  

\subsection{}There is a general process of \emph{rigidification} \cite[\S 5.1]{ACV} that we can use to remove the $\mathbb{G}_m$ from the automorphism (see \cite[p. 76]{MT1}).   We will refer to the rigidified stack as $\mathcal{M}$.  Note that in \cite{MT1} this notation is used for the open substack where furthermore all vectors in a splitting type are linearly independent.  As we want to consider non-simplical fans below, we shall not make this extra restriction here (this is the only instance where our notation will deviate from that of \cite{MT1}).  This stack $\mathcal{M}$ is still $G\times G$-equivariant, and honestly contains $G$ as an open substack.

\section{The stability condition}
\subsection{}There is no hope for the stack $\mathcal{M}$ to have a coarse (or good) moduli space.  We will however now formulate stability conditions which cut out open sub-stacks that do have these.  

Begin by choosing a fan  $\Sigma$ (in the sense of toric varieties) with support in the positive Weyl chamber of $G$.  In fact we want a little bit more: we also want a choice of an non-zero element of $\Lambda$ on each ray, which we will refer to as a \emph{ray vector}, i.e. we want to have a (torsion-free) \emph{stacky} fan in the sense of \cite{BCS, FMN}.  Strictly speaking stacky fans have to be simplicial, and \cite{MT1} was restricted to this case as well; we will show below in Section \ref{nonsimp} that this can be relaxed to arbitrary fans, but for now we shall restrict ourselves to the simplicial case.  Finally we need to make a choice of ordering of the ray vectors $\beta_1,\ldots,\beta_n$.  Absolutely any ordering is fine, see Section \ref{orderings} below for more comments about this.

\subsection{}With such a choice of stacky fan and ordering of the ray vectors, we can now give the corresponding stability condition: a framed bundle chain is stable if its isomorphism class is determined by a tuple of co-characters that all lie in the same Weyl chamber (which can be chosen to be the postive Weyl chamber), who moreover are all distinct ray vectors of a single cone in the fan, and occur in the chosen order from $p_+$ to $p_-$.  See Figure \ref{pic1} for an illustration.

\begin{figure}[h]
\begin{center}
\begin{tikzpicture}

\newcommand{\dubbellinks}[1]
{\draw[line width=1pt](0,0)-- (.5,2.4); 
\draw[line width=1pt](.5,1.6)-- (0,4); 
\filldraw (.078,.4) circle (2pt);
\filldraw (.078,3.6) circle (2pt);
\draw (.15, .4) node[anchor=  east] {};
\draw (.15, .4) node[anchor=  west] {};
\draw (.15, 3.6) node[anchor=  east] {};
\draw (.15, 3.6) node[anchor=  west] {};
\draw ( .4,2) node[anchor = west]{#1};
}

\newcommand{\trippel}[2]
{
\draw[line width=1pt](0,0)--(-.5,1.6);
\draw[line width=1pt](-.5,1.2)--(0,2.8);
\draw[line width=1pt](0,2.4)--(-.5,4);
\filldraw (-0.13,.4) circle (2pt);
\filldraw (-0.37,3.6) circle (2pt);
\draw (-.1, .4) node[anchor=  east] {};
\draw (-.1, .4) node[anchor=  west] {};
\draw (-.3, 3.6) node[anchor=  east] {};
\draw (-.3, 3.6) node[anchor=  west] {};
\draw (-.4,1.466) node[anchor=west]{#2};
\draw (-.1,2.533) node[anchor=west]{#1};
}

\begin{scope}[xshift=-1.8cm, yshift=0cm, scale=1.3]

\fill[gray!30!white] (0,0) -- (3,0) .. controls (3,1) and (2.33, 2.1) .. (1.5, 2.6) -- cycle ;
\foreach \x in {-2,...,2}
\foreach \y in { -1.73, 0, 1.73}
 \foreach \z in {-2.5, -1.5,...,2.5}
\foreach \w in {-.866,.866}
{\filldraw[gray!60!white]  	(\x  , \y ) 		circle 	(1.5pt);
\filldraw[gray!60!white]   (\z , \w )	circle 	(1.5pt);}

\filldraw[gray!60!white]  	(.5  , 2.6 ) circle 	(1.5pt);
\filldraw[gray!60!white]	(-.5  , 2.6 ) circle 	(1.5pt);
\filldraw[gray!60!white]  	(-.5  , -2.6 ) circle 	(1.5pt);
\filldraw[gray!60!white]  	(.5  , -2.6 ) circle 	(1.5pt);
\filldraw[gray!60!white]  	(1.5  , 2.6 ) circle 	(1.5pt);
\filldraw[gray!60!white] 	(-1.5  , 2.6 ) circle 	(1.5pt);
\filldraw[gray!60!white]  	(1.5  , -2.6 ) circle 	(1.5pt);
\filldraw[gray!60!white]  	(-1.5  , -2.6 ) circle 	(1.5pt);
\filldraw[gray!60!white]  	(-3  , 0 ) circle 	(1.5pt);
\filldraw[gray!60!white]  	(3  , 0 ) circle 	(1.5pt);

\filldraw [black] 	(1,1.73 ) circle 	(1.5pt);
\filldraw [black] 	(2,1.73 ) circle 	(1.5pt);
\filldraw [black] 	(1,0 ) circle 	(1.5pt);
\filldraw [black] 	(-.5,-.866 ) circle 	(1.5pt);
\filldraw [black] 	(-2.5,-.866) circle 	(1.5pt);
\filldraw [black] 	(-2,0 ) circle 	(1.5pt);

\draw[very thick] (-3,0) -- (3,0);
\draw[very thick] (-1.5,-2.6) -- (1.5,2.6);
\draw[very thick] (1.5,-2.6) -- (-1.5,2.6);
\draw[dashed] (0,0) -- (2.25,1.946);
\draw[dashed] (0,0) -- (2.25,-1.946);
\draw[dashed] (0,0) -- (0.5625,2.925);
\draw[dashed] (0,0) -- (0.5625,-2.925);
\draw[dashed] (0,0) -- (-2.8125,.97425);
\draw[dashed] (0,0) -- (-2.8125,-.97425);

\draw (1,1.73) node[anchor = east] {$\scriptstyle{\beta_1}$};
\draw (2,1.73) node[anchor = west] {$\scriptstyle{\beta_2}$};
\draw (1,0) node[anchor = north] {$\scriptstyle{\beta_3}$};
\draw (-2,0) node[anchor = north] {$\scriptstyle{w\beta_1}$};
\draw (-2.5,-.866) node[anchor = north] {$\scriptstyle{w\beta_2}$};
\draw (-.5,-.866) node[anchor = east] {$\scriptstyle{w\beta_3}$};
\end{scope}

\begin{scope}[xshift=-8.5cm,yshift=-9cm, scale=.8]
\draw (-1,5) node[anchor = west] {\phantom{xxxx}\small $\Sigma$-stable bundle chains include};

\begin{scope}[xshift=0cm, yshift= 2cm]
\draw[line width=1pt](0,2)--(0,-2);
\filldraw (0,-1.6) circle (2pt);
\draw (0.1, -1.6) node[anchor=  east] {};
\draw (0.1, -1.6) node[anchor=  west] {};
\filldraw (0,1.6) circle (2pt);
\draw (0.1, 1.6) node[anchor=  east] {};
\draw (0.1, 1.6) node[anchor=  west] {};
\end{scope}

\begin{scope}[xshift=1.5cm,yshift=0]
\dubbellinks{$\scriptstyle{\beta_1}$}
\end{scope}

\begin{scope}[xshift=3.2cm, yshift=0]
\dubbellinks{$\scriptstyle{w\beta_3}$}
\end{scope}

\begin{scope}[xshift=5.5cm,yshift=0]
\trippel{$\scriptstyle{\beta_2}$}{$\scriptstyle{\beta_3}$}
\end{scope}

\begin{scope}[xshift=7.3cm,yshift=0]
\trippel{$\scriptstyle{w\beta_1}$}{$\scriptstyle{w\beta_2}$}
\end{scope}

\end{scope}

\begin{scope}[xshift=.5cm,yshift=-9cm, scale=.8]
\draw (-2,5) node [anchor = west]{\phantom{xxxx}\small $\Sigma$-unstable bundle chains include};

\begin{scope}[xshift=0cm,yshift=0]
\trippel{$\scriptstyle{\beta_2}$}{$\scriptstyle{\beta_1}$}
\end{scope}

\begin{scope}[xshift=2cm,yshift=0]
\trippel{$\scriptstyle{\beta_1}$}{$\scriptstyle{\beta_3}$}
\end{scope}

\begin{scope}[xshift=4cm,yshift=0]
\trippel{$\scriptstyle{\beta_2}$}{$\scriptstyle{w\beta_3}$}
\end{scope}

\begin{scope}[xshift=5.5cm, yshift= 0cm]
{\draw[line width=1pt](0,0)-- (.5,1.2); 
\draw[line width=1pt](.5,0.8)-- (0,2.2); 
\draw[line width=1pt](0,1.8)-- (.5,3.2); 
\draw[line width=1pt](.5,2.8)-- (0,4); 
\filldraw (.165,.4) circle (2pt);
\filldraw (.165,3.6) circle (2pt);
\draw (.22, .4) node[anchor=  east] {};
\draw (.22, .4) node[anchor=  west] {};
\draw (.22, 3.6) node[anchor=  east] {};
\draw (.22, 3.6) node[anchor=  west] {};}
\draw (.4, 3) node[anchor=  west] {$\scriptstyle{\beta_1}$};
\draw (0, 2) node[anchor=  west] {$\scriptstyle{\beta_2}$};
\draw (.4, 1) node[anchor=  west] {$\scriptstyle{\beta_3}$};
\end{scope}

\end{scope}

\end{tikzpicture}\\ \
\end{center}
\caption{Illustration of a stability condition determined by a stacky fan}\label{pic1}
\end{figure}
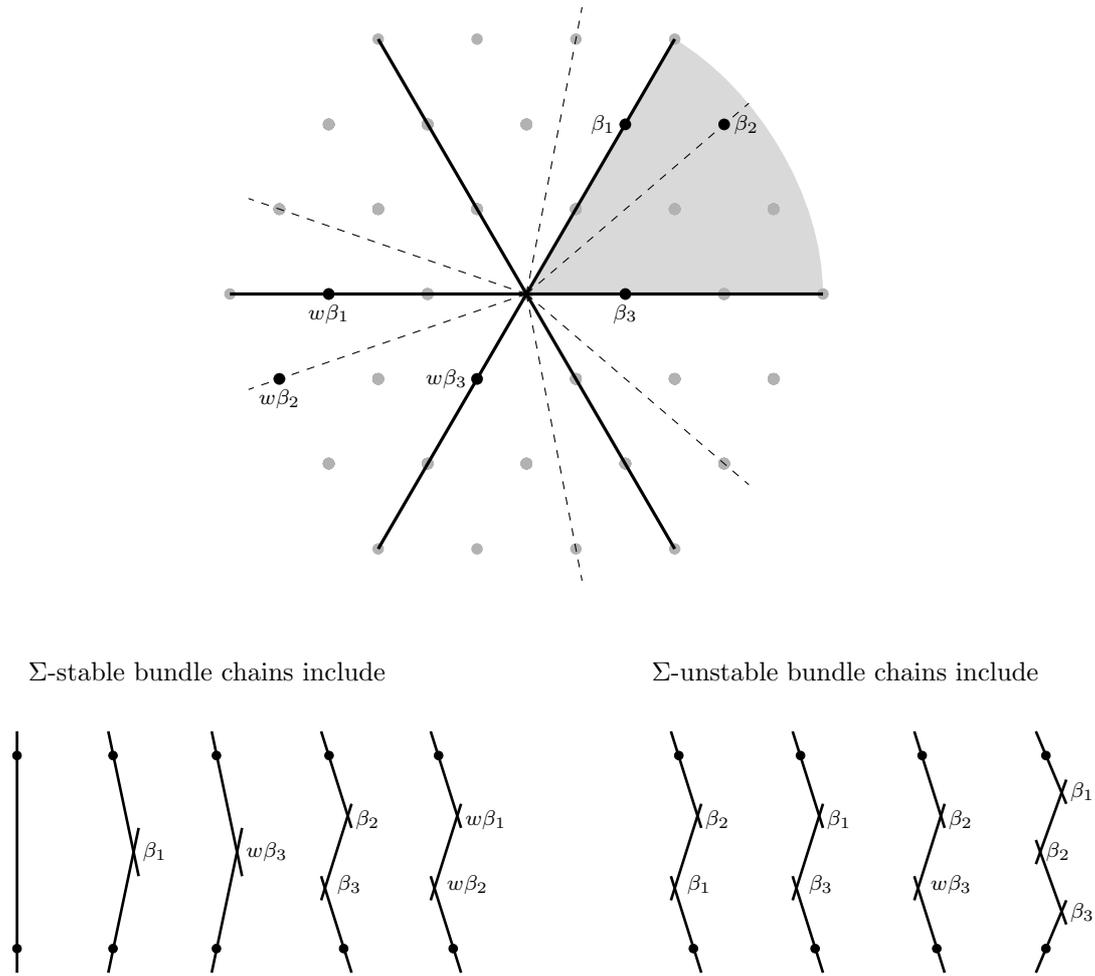

Imposing the stability condition cuts out an open substack $\mathcal{M}_G(\Sigma)\subset\mathcal{M}$.
The main result of \cite{MT1} is now the following: 

\begin{theorem}
The stack $\mathcal{M}_G(\Sigma)$ is a smooth, separated tame stack (hence Deligne-Mumford if $k$ has characteristic zero), which is proper if the support of $\Sigma$ is all of the Weyl chamber.  The coarse moduli space of $\mathcal{M}_G(\Sigma)$ is the toroidal spherical embedding of $(G\times G)/G$ given by the uncolored fan $w\Sigma$, where $w$ is the longest element in $W$.
\end{theorem}

Note that though the stability condition was dependent on the choice of ordering of the ray vectors $(\beta_1, \ldots , \beta_n)$ in the stacky fan, for different choices the resulting stacks are canonically isomorphic.

\section{The Cox-Vinberg quotient construction}
In this section we will show that the stacks $\mathcal{M}_G(\Sigma)$ discussed above can actually be obtained as quotients of a quasi-affine variety by a torus.  This quasi-affine variety is an open subvariety of a reductive monoid, and if the fan $\Sigma$ is polar (i.e. coming from a polytope) such that $W\Sigma$ has convex support, it can be understood as a \mbox{(semi-\!)} stable locus in the sense of Geometric Invariant Theory (which guarantees that the corresponding coarse moduli space is projective over an affine, and indeed projective if the fan $W\Sigma$ was complete).

\subsection{The Cox construction revisited}
The first ingredient that we need is the classical Cox constructions for toric varieties.  Let us briefly recall how this goes: suppose $X$ is a toric variety for a torus $T$, given by a fan $\Sigma$ with $N$ rays.  As before we assume ray vectors $\beta_i\in\Lambda$ chosen; if one is only interested in toric varieties one can simply chose the first non-zero elements of $\Lambda$ along the rays of $\Sigma$ for these.  

The standard presentation of the Cox construction assumes that the $\beta_i$ generate $\Lambda\otimes_{\mathbb{Z}}\mathbb{Q}$ (this is always the case if $X$, or equivalently $\Sigma$, is complete).  In this case we have a short exact sequence (where $\mathbb{G}_{\beta}:=\mathbb{G}_m^N$)

\begin{equation}\label{SES}
\begin{tikzcd}
1\ar[r] &L\ar[r] &\mathbb{G}_{\beta} \ar[rr,"{(\beta_1,\dots,\beta_N)}"] &
&T\ar[r] &1.
\end{tikzcd}
\end{equation}

Remark that this sequence only depends on the ray(vector)s of $\Sigma$, not on the full fan structure.  Cox now defines an open subvariety $\mathbb{A}^0_{\beta}$ inside $\mathbb{A}_{\beta}:=\mathbb{A}^N$ as the union of the $U_{\sigma}$, where for every cone $\sigma$ of $\Sigma$ we define $$U_{\sigma}=\{ (x_1,\ldots, x_N)\in \mathbb{A}_{\beta}\ |\ x_i\neq 0 \text{ if } \beta_i \text{ is not a ray vector of }\sigma\}$$ (so  $\mathbb{A}_{\beta}^0$ does depend on all the data in the fan). 

Cox shows that the toric variety $X$ can be understood as the categorical quotient $\mathbb{A}_{\beta}^0/L$, which is a geometric quotient when $\Sigma$ is simplicial.  Moreover, if $\Sigma$ is the normal fan to a polyhedron (which corresponds to $X$ being projective over an affine variety), then $\mathbb{A}^0_{\beta}/L$ can be understood as the semi-stable locus in the sense of geometric invariant theory (in fact stable if $\Sigma$ is simplicial).  Finally, if $\Sigma$ is simplicial, then also for the corresponding toric orbifolds we have $X\cong [\mathbb{A}_{\beta}^0/K]$.

\subsection{}\label{cox-altern}
Now, it turns out that the condition that the $\beta_i$ generate $\Lambda\otimes \mathbb{Q}$ can be relaxed in an easy way: even if this is not the case, i.e. when the sequence (\ref{SES}) is no longer exact on the right (for instance in the extreme case where $X=T$, where there are no $\beta_i$ at all), we still have that $$X\cong \left(\mathbb{A}^0_{\beta}\times T\right)/\mathbb{G}_{\beta},$$ and all the above statements (about geometric vs categorical quotients, orbifolds, and GIT) remain valid in this generality. 
Whereas $\mathbb{A}_{\beta}$ was $\Spec$ of the Cox ring of $X$, $\mathbb{A}_{\beta}\times T$ is $\Spec$ of the \emph{equivariant} Cox ring -- using the equivariance in this way is exactly what compensates for the lack of compactness.

So this generalizes the Cox construction to all toric varieties, at no additional cost, but more importantly for us this  variation on the Cox constrution also has a non-abelian generalization.  For this we need an extra ingredient: the Vinberg monoid (a.k.a. enveloping semigroup) of $G$.

\subsection{The Vinberg monoid}

The Vinberg monoid $S_G$, introduced in \cite{vinberg}, is a reductive monoid (i.e. a monoid object in the category of (affine) algebraic varieties, whose group of units is a reductive group) that one can canonically associate to a reductive group $G$.  Its group of units is $G_{\enh}=(G\times T) /Z_G $, and it comes with a canonical monoid morphism to the affine space $\mathbb{A}=S_G/\!\!/ (G\times G)$, which can be understood as the toric variety for the group $T/Z_G$ corresponding to the positive Weyl chamber (note that this is indeed an affine space, as the primitive elements along the rays are the fundamental co-weights, which indeed generate the co-weight lattice).  

At least for $k$ algebraically closed of characteristic zero it can be constructed explicitly, as $\Spec$ of a subring of the ring of regular functions of $G_{\enh}$.  By the algebraic Peter-Weyl theorem, the latter can be decomposed as $$k[G_{\enh}]=\bigoplus_{(\lambda,\mu)\in V^+_{G_{\enh}}} k[G_{\enh}]_{(\lambda,\mu)},$$ where $k[G_{\enh}]_{(\lambda,\mu)}$ denotes the isotypical components of type $(\lambda,\mu)$ (which one can think of as the matrix coefficient for the corresponding highest weight representations), and the character lattice of $G_{\enh}$ can be described as $$V_{G_{\enh}}=\Big\{(\lambda,\mu)\in V_G^2\ |\ \mu-\lambda = \sum_i m_i \alpha_i, \text{ with all }m_i \in \mathbb{Z} \Big\}$$ (here the $\alpha_i$ are the positive simple roots of $G$).  In this language we have $$k[S_G]=\bigoplus_{(\lambda,\mu)\in V^+_{G_{\enh}}\cap Q_G} k[G_{\enh}]_{(\lambda,\mu)}, $$ where $Q_G$ is the cone consisting of those $(\lambda,\mu)$ such that $\mu-\lambda=\sum_i m_i\alpha_i$ where all $m_i\geq 0$.  A description of $S_G$ in arbitrary characteristic (for $k$ algebraically closed) was given in \cite{rittatore}.

Vinberg showed (for $G$ semisimple) that $S_G$ satisfies a universal property in the category of reductive monoids \cite[Theorem 5]{vinberg}: essentially it says that any reductive monoid $S$ satisfying some natural conditions is obtained from $S_G$, where $G=[S^{\times}, S^{\times}]$ (here $S^{\times}$ is the group of units of $G$), as the base change of $S_G\rightarrow \mathbb{A}$ by a toric morphism of an affine toric variety to $\mathbb{A}$.  
From our point of view it is useful to think of $S_G$ as a (universal) degeneration of $G$ (the generic fibres of $S_G\rightarrow \mathbb{A}$ are indeed isomorphic to $G$).

Vinberg studied $S_G$ extensively; in particular he introduced a certain open subvariety $S_G^0$ (not a submonoid) of $S_G$, and showed that the geometric quotient $S_G^0/T$ is the wonderful compactification of $G_{\text{ad}}=G/Z$ of  De Concini-Procesi \cite{DCP1}.  It was shown in \cite[Theorem 5.4]{MT1} that $S_G^0$ can also be understood as a GIT-stable set for a suitable linearization of the action of $T$ on $S_G$. 

 \subsection{The Cox-Vinberg hybrid}
 It turns out that the Cox construction (or rather the variation thereof sketched above) and the Vinberg construction of the wonderful compactification of $G_{\text{ad}}$ can be hybridized in a very natural way.  Indeed, as all of the $\beta_i$ are in the positive Weyl chamber of $G$, the morphisms $\beta_i:\mathbb{G}_m\rightarrow T$ naturally induce monoid morphisms $\mathbb{A}^1\rightarrow \mathbb{A}$.  These combine into a monoid morphism $\mathbb{A}_{\beta}\rightarrow \mathbb{A}$, and  we can now look at the reductive monoid $S_{G,\beta}$  arising as the fibred product $\mathbb{A}_{\beta}\times_{\mathbb{A}} S_G$.  Inside of $S_{G,\beta}$ we can now again look at an open subvariety $S^0_{G,\beta}:=\mathbb{A}_{\beta}^0\times_{\mathbb{A}} S_G^0$.  If $\Sigma$ is a polar fan this can again be understood to be a GIT semi-stable locus  \cite[Theorem 8.1]{MT1} (stable if $\Sigma$ is simplicial).  In \cite[proof of Theorem 6.4]{MT1} a family of framed bundle chains was constructed over $\mathbb{A}_{\beta}\times_{\mathbb{A}} S_G^0$, which induces an isomorphism of stacks $$\left[S_{G,\beta}^0/\mathbb{G}_{\beta}\right]\cong \mathcal{M}_G(\Sigma).$$

\subsection{}In some sense one could think of this Cox-Vinberg monoid $S_{G,\beta}$ as playing a similar rol to $\mathcal{M}_G(\Sigma)$ as the quot-scheme plays to the moduli space of bundles on a smooth projective curve.  There is a crucial difference however: we don't know what moduli problem these $S_{G,\beta}$ represent, and in fact are only able to construct a family of framed bundle chains over the open subvariety $\mathbb{A}_{\beta}\times_{\mathbb{A}}S_{G}^0$.  It is very tempting to search for a modular interpretation of all of the monoid; see Section \ref{modularvinberg} below.

\section{The Non-Simplicial / Artin case}\label{nonsimp}

\subsection{}As mentioned above, the treatment in \cite{MT1} was restricted to stability conditions given by \emph{simplicial} (stacky) fans, which lead to moduli stacks that are smooth Deligne-Mumford or tame.  We will show here that with little extra effort one can remove the simplicial condition, leading to moduli stacks that are Artin but still have so-called \emph{good} moduli spaces, in the style of Alper \cite{alper}.  To this end we shall from now on understand the notion of a stacky fan $\Sigma$ to mean an arbitrary fan (i.e. not necessarily simplicial), together with a non-zero ray vector $\beta_i\in\Lambda_G$ on each  of its rays, as well as an ordering of these ray vectors.

The approach we shall take is to generalise the Cox-Vinberg construction.  This is not the only possible path, but in the style of this note we prefer it as it extends the results in the most economical fashion.  As we want to use the gluing results of good moduli spaces from \cite{alper}, we restrict ourselves now to working over algebraically closed fields of characteristic zero.  Presumably the gluing can be extended to the setting of \emph{adequate} moduli spaces, as in \cite{alper2}, in which case the restriction on characteristic is not necessary.  

Recall also that the construction of the Vinberg monoid over algebraically closed fields of positive characteristic was done by Rittatore in \cite{rittatore} using the Luna-Vust theory of spherical embeddings.  It is widely assumed by the experts that the Vinberg monoid also exists for split reductive groups over arbitrary fields, though we are not aware of any place where this is written down.  The theory of spherical embeddings has been extended by Huruguen \cite{huruguen} to more general settings (including the case of split reductive groups over arbitrary fields); presumably the combination of the work of Huruguen and Rittatore would extend the notion of Vinberg monoid to arbitrary fields, which should also extend the discussion below.

\subsection{}The notion of a good moduli space \cite{alper} for an Artin stack generalizes the relationship between the stack $[X^{\operatorname{ss}}/G]$ and the GIT-quotient $X/\!\!/G$ for the linearized action of a reductive group on a (semi-)projective variety -- here $X^{\operatorname{ss}}$ stands for the semi-stable locus in $X$.  For a general Artin stack $\mathcal{X}$ a good moduli space is an algebraic space $Y$, together with a morphism $\phi:\mathcal{X}\rightarrow Y$ such that the push-forward functor on quasi-coherent sheaves for $\phi$ is exact, and such that the induced morphism on sheaves $\mathcal{O}_Y\rightarrow \phi_*\mathcal{O}_{\mathcal{X}}$ is an isomorphism.  If a good moduli space exists it is unique, in the sense that $\phi$ is universal for maps from $\mathcal{X}$ to algebraic spaces.  

Of importance for us is that good moduli spaces can be glued, in the following sense \cite[\S 7.8]{alper}: if $\phi:\mathcal{X}\rightarrow Y$ is an Artin stack with a good moduli space, an open substack $\mathcal{U}\subset \mathcal{Y}$ is said to be \emph{saturated} if $\phi^{-1}(\phi(\mathcal{U}))=\mathcal{U}$.    We now have 

\begin{proposition}[{\cite[Proposition 7.9]{alper}}]\label{alpgluing}
Suppose $\mathcal{X}$ is an Artin stack covered by open substacks $\mathcal{U}_i$ such that for each $i$ there exist a good moduli space $\phi_i:\mathcal{U}_i\rightarrow Y_i$, with $Y_i$ a scheme.  Then there exist a good moduli space $\phi: \mathcal{X}\rightarrow Y$ and open sub-algebraic spaces $\tilde{Y_i}\subset Y$ such that $\tilde{Y}_i\cong Y_i$ and $\phi^{-1}(\tilde{Y}_i)=\mathcal{U}_i$ if and only if for each $i,j$, $\mathcal{U}_i\cap \mathcal{U}_j$ is saturated for $\phi:\mathcal{U}_i\rightarrow Y_i$.  

\end{proposition}

\subsection{}
We want to apply Proposition \ref{alpgluing} to the open covering of $\mathcal{M}_G(\Sigma)$ given by the cones $\sigma$ in a stacky fan $\Sigma$ (of course we want to think of these as stacky fans in their own right, so we remember the ray vectors $\beta_i$).   We shall refer to these open substacks as $\mathcal{M}_G(\sigma)$. The criterion will be easy to verify, since we know, by the results on the GIT construction in \cite{alper}, that they hold for those stability conditions that give rise to semi-projective good moduli spaces.

\subsection{}
The conditions under which the stability condition that cuts out $\mathcal{M}_G(\Sigma)$ can be understood as GIT-(semi-)stability for the Cox-Vinberg monoid were described in \cite[\S 8]{MT1}.  Note that we are applying GIT here to all \emph{semi-projective} schemes, i.e.  schemes projective over an affine.  These include all affine and projective varieties, and more precisely are exactly the schemes that can be generated as $\Proj$ of a graded $k$-algebra.  We shall say that a fan is \emph{normal} if it is determined by some polyhedron (with the ray vectors of the fan the outward normal vectors of the polyhedron).  

\begin{theorem}[{\cite[Theorem 8.1]{MT1}}]  If $\Sigma$ is a normal fan such that $W\Sigma$ has convex support then the stack $\mathcal{M}_G(\Sigma)$ is isomorphic to the stack $[S_{G,\beta}^{\operatorname{ss}}/ \mathbb{G}_m]$, where $S_{G,\beta}^{\operatorname{ss}}$ is the GIT-semistable locus, which equals $S_{G,\beta}^0$.
\end{theorem}
Observe that strictly speaking in \cite{MT1} only the case of simplicial fans was discussed (and hence no properly semi-stable points were present); the proof of the above theorem, however, goes through without change in the general case.  Also note that the condition for a fan $\Sigma$ with support in the postive Weyl chamber to be normal such that $W\Sigma$ has convex support can also be stated as saying that $\Sigma$ is the normal fan to a polyhedron $P$ in the positive Weyl chamber such that $WP$ is a polytope which meets the walls of all Weyl chambers perpendicularly.  

\begin{lemma}\label{contained}
Any (stacky) cone $\sigma$ with support in the positive Weyl chamber of $G$ can be included in a normal fan $\Sigma$ such that $W\Sigma$ has convex support.
\end{lemma}
\begin{proof}
First, remark that without loss of generality, we can assume that $\sigma$ has dimension equal to the rank of $G$.  Let $\widetilde{P}$ be the polyhedron (in $V_{\mathbb{R}}$) that is obtained by taking $-\sigma^{\vee}$ (where $\sigma^{\vee}$ is the dual cone to $\sigma$), and shifting it by a non-zero-vector $\overline{v}$ contained in $\sigma^{\vee}$.  The issue is that $\widetilde{P}$ may not meet the walls of the Weyl chambers perpendicularly.  In order to correct this, let $Q$ be the polyhedron given by taking the cone whose outward normal vectors are the simple positive roots of $G$ (which will not be strongly convex if $G$ is not semi-simple), and shifting it by $(1+\epsilon)\overline{v}$.  If we now take $\epsilon>0$ small enough, the polyhedron given by the intersection $P=\widetilde{P}\cap Q$ will now meet all of the walls of the Weyl chamber perpendicularly, and its polar fan (wich is independent of the choices of $\overline{v}$ and $\epsilon$ under the stated conditions) will be contained in the positive Weyl chamber in $\Lambda_{\mathbb{R}}$, and it will contain the original cone $\sigma$.
\end{proof}

\subsection{}
We can now conclude: 

\begin{theorem}
For any stacky fan $\Sigma$ as above (possibly non-simplicial), we have that $\mathcal{M}_G(\Sigma)$ has a good moduli space, which is a scheme that as a (toroidal) spherical variety is given by the uncolored fan $w\Sigma$.
\end{theorem}
 
\begin{proof}
We need to verify that the conditions of Proposition \ref{alpgluing} hold for the substacks $\mathcal{M}_G(\sigma)$, for any (stacky) cone $\sigma$ in $\Sigma$.  So firstly, we need to establish that each $\mathcal{M}_G(\sigma)$ has a good moduli space.  In light of \cite[Remark 6.2]{alper}, this boils down to verifying that each $\mathcal{M}_G(\sigma)$ is saturated in $\mathcal{M}_G(\widetilde{\Sigma})$, for some stacky fan $\widetilde{\Sigma}$ containing $\sigma$ such that $\mathcal{M}_G(\widetilde{\Sigma})$ has a good moduli space.  By Lemma \ref{contained} we can always find such a $\widetilde{\Sigma}$, and moreover assume $\mathcal{M}_G({\widetilde{\Sigma}})$ (or rather its good moduli space) is obtained as a GIT-quotient; hence we can apply the general results of GIT there.  In particular, we know that in the good moduli space orbits need to be identified if their closures meet in the semi-stable locus, in casu $S_{G,\beta}^0$.  

We can now use the Cox-Vinberg construction to show that  $\mathcal{M}_G({\sigma})$ will be saturated in $\mathcal{M}_G(\widetilde{\Sigma})$, which means that no two orbits that do not lie in $U_{\widetilde{\sigma}}\times_{\mathbb{A}} S_G^0$ for some cone $\widetilde{\sigma}$ in $\widetilde{\Sigma}$ will have closures that intersect in $S_{G,\beta}^0$.  Indeed, again using \cite[Remark 6.2]{alper}, observing that the corresponding result holds for toric varieties (since any toric variety corresponding to a single cone can itself be obtained as a GIT-quotient -- something not true for arbitrary toroidal spherical varieties corresponding to a single cone),  and hence using the variation on the Cox construction described above in Section \ref{cox-altern} we find that if orbits in $\mathbb{A}_{\beta}^0\times T$ intersect, they have to be contained in $U_{\widetilde{\sigma}}\times T$ for some cone $\widetilde{\sigma}$ in $\widetilde{\Sigma}$.  As all orbits in $T$ itself are closed, we find moreover that the corresponding orbit closures intersect in $U_{\widetilde{\sigma}}$. 

We can now go back to $S_{G,\beta}^0$: if orbit closures intersect there, so will their images in $\mathbb{A}_{\beta}^0$, and hence also the corresponding orbits in $\mathbb{A}_{\beta}^0\times T$.  Using the toric statement, we find that the original orbits had to be contained in some $U_{\widetilde{\sigma}}\times_{\mathbb{A}} S_G^0$.  Hence $[U_{\sigma}\times_{\mathbb{A}}S_G^0 \ \big/ \ \mathbb{G}_{\beta}]$ is saturated in $\mathcal{M}_G(\widetilde{\Sigma})$, and therefore each $\mathcal{M}_G(\sigma)$ has a good moduli space. 

The fact that for any $\sigma,\sigma'$ in $\Sigma$ we have that $\mathcal{M}_{G}(\sigma\cap \sigma')$ is saturated in $\mathcal{M}_G(\sigma)$ follows from a similar reasoning using a $\widetilde{\Sigma}$ as above that contains $\sigma$ (note that it is not necessary that $\sigma'$ be contained in $\widetilde{\Sigma}$, the only fact we use is that $\sigma\cap\sigma'$ is a sub-cone of $\sigma$).  

Finally the reasoning identifying the good moduli space of $\mathcal{M}_G({\Sigma})$ as the spherical embedding corresponding to $w\Sigma$ is exactly the same as \cite[Proposition 10.1]{MT1}.
\end{proof}

\section{Non-abelian symplectic cutting}\label{non-ab cut}
\subsection{}
In this section we will change viewpoint, and think about those embeddings $\mathcal{M}_G(\Sigma)$ (for $k=\mathbb{C}$) whose good moduli space is semi-projective in terms of (real) symplectic geometry.

To this end, recall that the classical (abelian) symplectic cut procedure, introduced by Lerman \cite{lerman}, takes a Hamiltonian $T$-space  $M$(for $T$ a compact connected torus) and a polyhedron $P\subset \mathfrak{t}^*$ (where $\mathfrak{t}=\Lie(T)$), the normal vectors to all of whose facets lie in the lattice $\Hom((1),T)\subset \mathfrak{t}^*$. It returns a new Hamiltonian $T$-space $M_P$, whose image under the moment map is $\mu_T(M_P)=\mu_T(M)\cap P$, and such that the pre-images of the interior of $P$ in $M$ and $M_P$ are $T$-equivariantly symplectomorphic.  

Symplectic cutting can be understood from two points of view, either as a global quotient $$M_P=(M\times \mathbb{C}^N)/\!\!/ U(1)^N,$$ (where $/\!\!/$ denotes symplectic reduction here), or as a local surgery, by taking the pre-image of $P$ under $\mu_T$, and collapsing the circle-actions determined by the normal vectors on the pre-images of the facets of $P$.  Note that in particular the global quotient construction allows the cutting to be interpreted in K\"ahler, or even algebro-geometric terms.  On the other hand, for the local surgery picture one just needs a map to a polytope such that the projection to the normal direction for each facet is a moment-map in a neighbourhood of that facet (which is strictly less than asking for a full $T$-moment map).

\subsection{}
Symplectic cutting is an elementary but very useful operation, and it is natural to try to generalise it to non-abelian group actions.  A number of suggestions for this are in the literature (e.g. \cite{weitsman, paradan}), but in some sense the most natural one is due to Woodward \cite{woodward} and further developed by Meinrenken \cite{meinrenken}.  This operation starts from a Hamiltonian $K$-space $M$, for $K$ a compact connected Lie group.  It is useful to compose the moment map $\mu_K:M\rightarrow \mathfrak{k}^*$ with the quotient map under the co-adjoint action $\mathfrak{k}^*\rightarrow \mathfrak{t}^*_+$ to the positive Weyl chamber in the dual of the Lie algebra of a maximal torus of $K$ -- we shall denote this composition $\Phi$.  If $P$ is now a polyhedron in $\mathfrak{t}^*_+$ such that if it meets any facet of $\mathfrak{t}^*_+$ in $\Phi(M)$, it does so perpendicularly, one can again apply the recipe for the local surgery for abelian cutting (since the components of $\Phi$ given by the outward normal vectors will all be smooth on neighborhoods of the pre-images of the corresponding faces), and apply it to this map $\Phi$.  The result will be a Hamiltonian $K$-space $M_P$, with $\Phi(M_P)=\Phi(M)\cap P$.  

One difference with the abelian set up is that a global (K\"ahler) quotient counter part to this construction seems to be missing: indeed,  Woodward applies his cut to a coadjoint orbit (which in particular is K\"ahler), to obtain a symplectic space that cannot carry any compatible K\"ahler structure \cite{woodward2}.  This seems to rule out any possibility of interpreting this operation in complex algebraic geometry.

We show in \cite{MT2} however that this is not the case, as long as the polytope $P$ has the property that all of its outward normal vectors lie in the positive Weyl chamber $\mathfrak{t}_+$ (and hence $P$ determines a fan $\Sigma$ as we have used above).  In this case a global quotient construction to the operation of Woodward and Meinrenken does exist.  It begins by  taking the symplectic reduction $(M\times S_G)/\!\!/ K$, where $S_G$ is the Vinberg monoid for the complexification $G=K_{\mathbb{C}}$ (one can canonically interpret $S_G$ as a (stratified) symplectic space by choosing a faithful representation $V$ of $S_G$ to embed it into the affine space $\End(V)$ equipped with the Hermitian K\"ahler structure -- the symplectic structure this endows $S_G$ with is independent of the choice of $V$).  This spaces comes with a natural morphism to $\mathbb{A}$, and the data of the polytope now allows us to take the fibered product with $\mathbb{A}_{\beta}$ over $\mathbb{A}$.  Finally the symplectic cut of Woodward and Meinrenken can now be constructed as a global quotient as 
$$M_P\cong \left(\left(\left(M\times S_G\right)/\!\!/ K\right) \times_{\mathbb{A}}\mathbb{A}_{\beta} \right) /\!\!/_{\xi} U(1)^N$$ (see \cite[Corollary 4]{MT2}). Note that the fibered product only needs the ray vectors of $\Sigma$, but the level at which the final symplectic reduction is performed uses all the data of $P$.  If $M$ was a complex semi-projective variety to begin with, and $P$ had all outward normal vectors in the positive Weyl chamber, this construction can therefore be understood to give a new semi-projective variety.

\subsection{}
 For details of the proof of this we refer to \cite{MT2}, but we will briefly mention the key ideas behind this: firstly, it is well-known that any suitably intrinsic operation $\mathfrak{X}$ one wants to apply to a Hamiltonian $K$-space $M$ can equivalently be done by first applying it to $T^*K$, and then taking the quotient of the carthesian product of the resulting space and $M$ -- symbolically we could write this as $$\mathfrak{X}(M)\cong \left(M \times \mathfrak{X}\left(T^*K\right)\right) /\!\!/ K.$$  This is the case for symplectic implosion \cite{gjs} where the implosion of $T^*K$ is dubbed the \emph{universal implosion} (and trivially for $\mathfrak{X}=\id$ or symplectic reduction by $K$), but also for symplectic cutting.  Indeed, in the abelian symplectic cut one has that the \emph{universal symplectic cut} of $T^*T$ is just the symplectic toric space corresponding to the polyhedron $P$.
  
 In fact, this latter point of view is just the symplectic counter part of the alternative approach to the Cox construction that was outlined above.  The standard symplectic counterpart to the Cox construction is known as the Delzant construction \cite{delzant}, which essentially says that any compact toric manifold (or one whose normal fan satisfies the condition that (\ref{SES}) is exact) can be obtained as the symplectic reduction of $\mathbb{C}^N$, where $N$ is the number of facets of $P$.  This again suffers from the same disadvantages as the standard Cox construction when one wants to consider more general toric symplectic spaces (with proper moment maps).  
 
 Whereas in the algebro-geometric picture this was remedied by looking at \emph{equivariant} Cox rings, in the symplectic picture one can just replace the symplectic reduction of $\mathbb{C}^N$ by the symplectic cut of $T^*T$ with respect to $P$ (one can think of this as the \emph{universal cut}).  Moreover, we find that by doing so we not only generalize the validity of the Delzant construction, but also allow for a non-abelian generalisation: indeed, it is shown in \cite[Theorem 3]{MT2} that the universal non-abelian cuts $T^*K_P$ correspond exactly to the spaces $\mathcal{M}_G(\Sigma)$ discussed above, and one can then use the Cox-Vinberg construction for the latter.
 
 The second idea, crucial in establishing this correspondence between $T^*K_P$ and $\mathcal{M}_G(\Sigma)$, is to use a \emph{section of the moment map} for $S_G$.  The use of such a section of the moment map was already implicit in \cite{delzant} (present whenever a square root showed up in the defining formulas), and can still be done in the non-abelian setting, by using the polar decomposition of complex matrices in combination with an (arbitrary) faithful representation of $S_G$.

\section{Open questions}
Finally, we will list here a number of open questions regarding the compactifications discussed above.
\subsection{Ordering of weights and shuffling functors}\label{orderings} As discussed above, the stability condition requires an \emph{ordering} of the ray vectors $\beta_i$.  Different orderings will cut out different substacks $\mathcal{M}_G(\Sigma)$, but it turns out these are all isomorphic, as can easily be seen from the Cox-Vinberg construction.  This implies that there is a sort of \emph{shuffling functor} that associates with every family of bundle chains that are stable for one choice of ordering, a new family (all of whose chains have the same length as before) that is stable with respect to another choice of ordering, such that the ray vectors in the splitting type of a bundle chain over a closed point just get shuffled around.

Other circumstantial evidence comes from thinking about affine toric varieties: all of these have modular interpretations by the description above, but at the same time we also know that all of affine toric varieties are canonically abelian monoids, extending the group operation of the torus.  It is therefore natural to ask how this monoidal structure can be understood from the modular point of view.  The most obvious candidate is given by gluing two chains together, attaching $p_+$ from the first to $p_-$ of the second.  If one has a shuffling functor available, one can then shuffle the chain to have the ray vectors in the correct order, and contract any component of the chain whose two associated ray vectors are identical.  This hypothetical shuffling operation should correspond to the monoidal operation on the affine toric variety.

In spite of this evidence we do not know how to describe such a shuffling functor directly however.

\subsection{Modular interpretations of the Vinberg monoid}\label{modularvinberg}  Related to the gluing of chains is the question wether the Vinberg monoid itself represents a moduli problem.  As an embedding of $G_{\enh}$, $S_G$ is not toroidal, so doesn't fall under the framework discussed above.  The stack of all framed bundle chains has a monoidal structure by gluing chains together, but the condition that all elements of the splitting type lie in the same Weyl chamber typically gets lost while doing this, so even the construction of $\mathcal{M}$ in \cite{MT1} needs to be expanded to accommodate this.

\subsection{Stratifications of the stack $\mathcal{M}$}
The stability condition described above is introduced in a seemingly ad-hoc way, motivated by the Birkhoff-Grothendieck-Harder theorem and its generalization, and the use of fans in toric and spherical geometry.  Recently, however, more systematic studies have been made of notions of (in-)stability on stacks, by Heinloth \cite{heinl} and Halpern-Leistner \cite{HL}.   One can therefore ask if it is possible to recover the notions of stability we have introduced from a general $\Theta$-stratification following Halpern-Leistner.  Note that since the method above also gives rise to non-projective complete good moduli spaces, at least some of the methods for constructing stability functions in \cite{HL} would have to be extended.

\subsection{Beyond toroidal}
The work described above is restricted to moduli stacks whose good moduli spaces are \emph{toroidal} embeddings of $G$.  The theory of spherical varieties also gives many other embeddings, classified by colored cones.  Can the modular interpretation be extended to these as well?

\subsection{Gauge-theoretic description}
In Section \ref{non-ab cut} above a symplectic viewpoint on the compactifications of $G$ was given in terms of non-abelian cutting.  This does not directly come with a modular interpretation however.  

It is a seminal result in the theory of moduli spaces of bundles on projective curves that they can also be interpreted (as stratified real smooth spaces) as moduli spaces of flat connections, for the compact form of the structure group, on the underlying surface.  This is a special case of the Hitchin-Kobayashi correspondence, in this setting due to Donaldson \cite{donaldson}, but essentially originating with the work of Narasimhan-Seshadri \cite{narases}.  Given that the spaces we consider can be interpreted as moduli spaces of $G$-bundles in algebraic geometry, it is natural to ask if a similar interpretation exists, in gauge-theoretic, or at least differential-geometric terms?  

An intruiging hint comes from the work of Mundet i Riera and Tian \cite{mundet-tian}, who give a compactification of the moduli space of so-called \emph{twisted holomorphic maps} -- equivariant maps of a $U(1)$-principal bundle over a Riemann surface into a symplectic manifold with Hamiltonian $U(1)$ action.  In their compactification they see that, when a Riemann surface degenerates to one with nodal singularities, they have to allow the images of the various components to be disjoint of each other, but connected by paths of broken flow lines for the $U(1)$ moment map.  The moduli problem we discuss above can in some sense be interpreted as being given by twisted holomorphic maps of $\mathbb{P}^1$ with two marked points into $BG$, and it is tantalizing to try to interpret the chains of projective lines similarly as ($U(1)$-orbits of) broken flow lines.

\subsection{Compactification of universal moduli space}
As we indicated above, a historical motivation for studying group compactifications is the link with a compactification of the  \emph{universal} moduli space of bundles on curves (where both are allowed to vary), lying over the Deligne-Mumford compactification of the moduli space of curves.  This has been an open problem for several decades now; in the case of vector bundles various constructions are available in the literature (e.g. \cite{pand,schmitt, nagses, gies}), and in particular for the case of line bundles the literature is vast, but for general principal bundles much less is known -- for recent approaches to this based on compactifying loop groups, see the work of Solis \cite{solis1,solis2}.

 Note that compactifying the group is not sufficient to obtain a compactification of these moduli spaces, but any such compactification does give rise to a compactification of the group, by taking a fixed stable bundle on a nodal curve whose normalisation is connected, and creating a family of bundles on the nodal curve out of this parametrised by the group, by changing the gluing at the node.  One can therefore ask if a compacticiation of the universal moduli space of bundles on curves exists, such that the corresponding compactifications of the group (with attached modular interpretation) are those described above.

\subsection{Link with work of Kausz} 
As mentioned in the introduction, Kausz introduced a particular equivariant compactification of $GL(n)$, and gave it a modular interpretation (in terms of ``generalized isomorphisms" and ``back-and-forth morphisms", see \cite[\S5]{kausz}), motivated by the study of degenerations of moduli spaces of vector bundles on smooth curves.  As a variety it can be understood by embedding $GL(n)$ into $PGL(n+1)$, and then taking the closure in the wonderful compactication of the latter.  As a spherical variety this is toroidal, and hence the variety can also be understood as a moduli space of framed bundle chains (see \cite[\S12]{MT1}).  

It would be interesting to see if the corresponding morphism can be understood in terms of the two moduli problems.   In other work of Kausz some ideas similar to those described above are present (e.g. moduli spaces of stable maps into $BGL(r)$ in \cite{kausz2}, or in particular chains of projective lines, dubbed \emph{zollst\"ocke}, in \cite{kausz3}).  Neverthless we have not been able to find a link between the two modular interpretations of the compactifcation of $GL(r)$.

\end{document}